\documentclass[12pt]{amsart}
\usepackage{}

\usepackage{amsmath}
\usepackage{amsfonts}
\usepackage{amssymb}
\usepackage[all]{xy}           %xypic macro for latex2.09

\usepackage{bbding}
\usepackage{txfonts}
\usepackage{amscd}

\usepackage[shortlabels]{enumitem}
\usepackage{ifpdf}
\ifpdf
  \usepackage[colorlinks,final,hyperindex]{hyperref}
\else
  \usepackage[colorlinks,final,hyperindex,hypertex]{hyperref}
\fi
\usepackage{tikz}
\usepackage[active]{srcltx}

\makeatletter

\newtheorem{thm}{Theorem}[section]
\newtheorem{lem}[thm]{Lemma}

\newtheorem{pro}[thm]{Proposition}
\newtheorem{ex}[thm]{Example}
\newtheorem{rmk}[thm]{Remark}
\newtheorem{defi}[thm]{Definition}

\newcommand {\emptycomment}[1]{}

\newcommand{\be }{\begin{equation}}
\newcommand{\ee }{\end{equation}}

\newcommand{\huaL}{\mathcal{L}}

\newcommand{\frkB}{\mathfrak B}

\newcommand{\frkL}{\mathfrak L}

\newcommand{\br}[1]{   [ \cdot,    \cdot  ]   }

\newcommand{\gl}{\mathfrak {gl}}

\newcommand{\ad}{\mathrm{ad}}

\begin{document}

\title[$F$-manifold color algebras]{$F$-manifold color algebras}

\author{Ming Ding, Zhiqi Chen and Jifu Li}
\address{Department of Mathematics and information science, Guangzhou University, Guangzhou,  P.R. China}
\email{m-ding04@mails.tsinghua.edu.cn}

\address{School of Mathematical Sciences and LPMC, Nankai University,
Tianjin, P.R. China}
\email{chenzhiqi@nankai.edu.cn}

\address{School of Science, Tianjin University of technology, Tianjin,  P.R. China}
\email{ljfanhan@126.com}
\vspace{-5mm}

%\date{\today}

\begin{abstract}
In this paper, we introduce the notion of $F$-manifold color algebras and study their properties which extend some results for $F$-manifold algebras.
\end{abstract}


\maketitle

\allowdisplaybreaks

%\end{document}
\section{Introduction}\label{sec:intr}
The notion of Frobenius manifolds was invented  by Dubrovin \cite{Dub95}  in order to give a geometrical expression of the Witten-Dijkgraaf-Verlinde-Verlinde equations.
  In 1999, Hertling and Manin \cite{HerMa} introduced the concept of $F$-manifolds as  a relaxation of the conditions of Frobenius manifolds. Inspired by the investigation of algebraic structures of $F$-manifolds, the notion of an $F$-manifold algebra is given by Dotsenko \cite{Dot} in 2019 to relate the operad $F$-manifold algebras to the operad  pre-Lie
algebras. An $F$-manifold algebra is defined as a triple $(A,\cdot,[ , ])$ satisfying the following Hertling-Manin relation,
 \[
P_{x\cdot y}(z,w)=x\cdot P_{y}(z,w) + y\cdot P_{x}(z,w),\ \ \ \forall x,y,z,w\in A,
 \]
where $(A,\cdot)$ is a commutative associative algebra, $(A,[ , ])$ is a Lie algebra and $P_{x}(y,z)=[x,y\cdot z]-[x,y]\cdot  z-y\cdot [x,z]$.
 A pre-Lie algebra is a vector space $A$ with a bilinear multiplication
$\cdot$ satisfying
 \[
 (x\cdot y)\cdot z-x\cdot (y\cdot z)= (y\cdot x)\cdot z-y\cdot (x\cdot z),\ \ \ \forall x,y,z\in A.
 \]
Pre-Lie algebras have attracted a great deal of attentions  in many areas of mathematics and physics (see {\cite{Bai,Ban,Burde,ChaLiv,DSV,MT} and so on).

Recently, Liu, Sheng and Bai \cite{LSB} introduced the concept of pre-$F$-manifold algebras which gives rise to $F$-manifold algebras. They also studied representations of $F$-manifold algebras, and constructed many other examples of $F$-manifold algebras.

In this paper, we introduce the notions of an $F$-manifold color algebra and a pre-$F$-manifold color algebra which can be viewed as natural generalizations of an $F$-manifold algebra and a pre-$F$-manifold algebra,  and then extend some properties of an $F$-manifold algebra in \cite{LSB}  to the color case.

The paper is organized as follows. In Section 2, we summarize some basic concepts of Lie color algebras, pre-Lie color algebras, representations of
$\varepsilon$-commutative associative algebras and Lie color algebras. In Section 3, we introduce the concept of  an $F$-manifold color algebra and study its
representations. We also introduce the notion of a coherence $F$-manifold color algebra and obtain that an $F$-manifold color algebra with a nondegenerate symmetric bilinear form is a coherence $F$-manifold color algebra. In Section 4, we introduce the concept of a  pre-$F$-manifold color algebra
and show that a pre-$F$-manifold color algebra can naturally give rise to  an $F$-manifold color algebra.

Throughout this paper, we assume that $\mathbb K$ is an algebraically closed field  of characteristic zero and all the vector spaces are finite dimensional  over  $\mathbb K$.

\section{Lie color algebras and relative algebraic structures}

The concept of a Lie color algebra was introduced by Scheunert  \cite{Scheu} which has been widely studied.
In this section, we collect some definitions and the representations  of Lie color algebras and some Lie color admissible algebras such as pre-Lie color
algebras and $\varepsilon$-commutative associative algebras. One can refer to  \cite{BP,CMN,CSO,F,MCL,NW,Scheu,ScheuZ,SCL,WZ} for more details.

\begin{defi}
 Let $G$ be an abelian group. A map $\varepsilon :G\times G\rightarrow \mathbb{K}  \backslash \{0\}$ is called a skew-symmetric bicharacter on $G$ if the following
identities hold,
\begin{enumerate}
 \item [(i)] $\varepsilon(a, b)\varepsilon(b, a)=1$,
\item [(ii)] $\varepsilon(a, b+c)=\varepsilon(a, b)\varepsilon(a, c)$,
\item [(iii)]$\varepsilon(a+b, c)=\varepsilon(a, c)\varepsilon(b, c)$,
\end{enumerate}
for all $a, b, c\in G$.
\end{defi}

By the definition, it is obvious that $\varepsilon(a, 0)=\varepsilon(0, a)=1, \varepsilon(a,a)=\pm 1 \;\mbox{for all}\; a\in G$.

\begin{defi}
Let $G$ be an abelian group with bicharacter $\varepsilon$ as above. The $G$-graded vector space
$$ A=\bigoplus_{g\in G }A_{g}$$ is called a pre-Lie color algebra, if $A$ has
a bilinear multiplication operation $\cdot$ satisfying
\begin{enumerate}
\item   $A_{a}\cdot A_{b}\subseteq A_{a+b},$
\item   $(x\cdot y)\cdot z-x\cdot (y\cdot z)=\varepsilon(a,b)((y\cdot x)\cdot z-y\cdot (x\cdot z)),$
\end{enumerate}
for all $x\in A_{a}, y\in A_{b}, z\in A_{c}$, and $a, b, c\in G$.
\end{defi}

\begin{defi}
Let $G$ be an abelian group with bicharacter $\varepsilon$ as above. The $G$-graded vector space
$$ A=\bigoplus_{g\in G }A_{g}$$ is called a Lie color algebra, if $A$ has
a bilinear multiplication  $[ , ]:A\times A \rightarrow A$ satisfying
\begin{enumerate}
\item   $[A_{a}, A_{b}]\subseteq A_{a+b},$
\item   $[x,y]=-\varepsilon(a,b)[y,x],$
\item   $\varepsilon(c,a)[x,[y,z]]+ \varepsilon(b,c)[z,[x,y]]+ \varepsilon(a,b)[y,[z,x]]=0,$
\end{enumerate}
for all $x\in A_{a}, y\in A_{b}, z\in A_{c}$, and $a, b, c\in G$.
\end{defi}
\begin{rmk}
It is well known that a pre-Lie algebra $(A, \cdot)$ with a commutator
$[x, y]=x\cdot y-y\cdot x$ becomes a Lie algebra.  Similarly, one has a pre-Lie color algebra's version, that is a pre-Lie color algebra
$(A, \cdot,\varepsilon)$ with a commutator $[x, y]=x\cdot y - \varepsilon(x, y) y\cdot x$ becomes a Lie color algebra.
\end{rmk}

An element $x$ of the $G$-graded vector space $A$ is called homogeneous of degree $a$ if $x\in A_{a}$.
In the rest of this paper, for homogeneous elements $x\in A_a$, $y\in A_b$, $z\in A_c$,
we will write $\varepsilon(x, y)$ instead of $\varepsilon(a, b)$,  $\varepsilon(x+y, z)$ instead of $\varepsilon(a+b, c)$, and so on.
Furthermore, when the skew-symmetric bicharacter $\varepsilon(x, y)$ appears, it means that $x$ and $y$ are both homogeneous elements.

By an $\varepsilon$-commutative associative algebra $(A, \cdot, \varepsilon)$, we mean that $(A, \cdot)$ is a $G$-graded  associative algebra
with the following $\varepsilon$-commutativity  $$x\cdot y=\varepsilon(x, y)y\cdot x$$
for all homogeneous elements $x, y\in A$.

Let $V$ be a $G$-graded vector space.  A representation $(V, \mu)$ of an $\varepsilon$-commutative associative algebra $(A, \cdot, \varepsilon)$ is a linear map
$\mu:A\longrightarrow \gl(V)$ satisfying
$$\mu(x)v\in V_{a+b},\ \ \ \ \mu(x\cdot y)=\mu(x)\circ\mu(y)$$
for all  $v\in V_{a}, x\in A_{b}, y\in A_{c}.$
A representation $(V, \rho)$ of a Lie color algebra $(A, [ , ], \varepsilon)$ is a linear map
$\rho:A\longrightarrow \gl(V)$ satisfying
$$\rho(x)v\in V_{a+b},\ \ \ \ \rho([x, y])=\rho(x)\circ\rho(y)-\varepsilon(x, y)\rho(y)\circ \rho(x)$$
for all  $v\in V_{a}, x\in A_{b}, y\in A_{c}.$

We consider the dual space $V^*$ of the $G$-graded vector space $V$. Then $V^*=\bigoplus_{a\in G }V^*_{a}$ is also a $G$-graded space, where
$$V^*_{a}=\{\alpha \in V^*| \alpha(x)=0,  b\ne -a, \forall x\in V_b, b\in G \}.$$

Define a linear map $\mu^*:A\longrightarrow \gl(V^*)$ by
$$
\mu^*(x)\alpha\in V^*_{a+c},\ \ \ \  \langle \mu^*(x)\alpha,v\rangle=-\varepsilon(x, \alpha)\langle \alpha,\mu(x)v\rangle$$
for all $x\in A_a,v\in V_b, \alpha\in V^*_c.$

  It is easy to see  that

  (1)\ If  $(V, \mu)$ is a representation of an $\varepsilon$-commutative associative algebra $(A, \cdot, \varepsilon)$, then $(V^*,-\mu^*)$ is a representation of $(A, \cdot, \varepsilon)$;

  (2)\  If $(V, \mu)$ is a representation of a Lie color algebra $(A, [ , ], \varepsilon)$, then $(V^*, \mu^*)$ is a representation of $(A, [ , ], \varepsilon)$.

\section{$F$-manifold color algebras and representations}
In this section, we will introduce the notion of $F$-manifold color algebras, and generalize some results in \cite{LSB} to the color case.
\begin{defi}
An $F$-manifold color algebra is a quadruple $(A,\cdot,[ , ], \varepsilon)$, where $(A,\cdot, \varepsilon)$ is an $\varepsilon$-commutative associative algebra and
$(A,[ , ], \varepsilon)$ is a Lie color  algebra, such that for all homogeneous elements $x,y,z,w\in A$, the color Hertling-Manin relation holds:
\begin{equation}\label{eq:HM relation}
P_{x\cdot y}(z,w)=x\cdot P_{y}(z,w)+\varepsilon(x,y) y\cdot P_{x}(z,w),
\end{equation}
where $P_{x}(y,z)$ is defined by
\begin{equation}
 P_{x}(y,z)=[x,y\cdot z]-[x,y]\cdot z-\varepsilon(x,y) y\cdot [x,z].
\end{equation}
\end{defi}

\begin{rmk}
When $G=\{0\}$ and $\varepsilon(0,0)=1$, then  an $F$-manifold color algebra is exactly an
$F$-manifold algebra.
\end{rmk}

\begin{defi}
Let $(A,\cdot,[ , ], \varepsilon)$  be an $F$-manifold color algebra. A representation of $A$ is a triple $(V,\rho,\mu)$,
 such that
$(V,\rho)$ is a  representation of the Lie color algebra $(A,[ , ], \varepsilon)$ and $(V,\mu)$ is a representation of the
 $\varepsilon$-commutative associative algebra $(A,\cdot, \varepsilon)$ satisfying
   \begin{eqnarray*}
       R_{\rho,\mu}(x_{1}\cdot x_{2},x_{3})=\mu(x_{1}) R_{\rho,\mu}(x_{2},x_{3})+\varepsilon(x_{1},x_{2})\mu(x_{2})R_{\rho,\mu}(x_{1},x_{3}),\\
      \mu(P_{x_{1}}(x_{2},x_{3}))=\varepsilon(x_{1},x_{2}+x_{3})S_{\rho,\mu}(x_{2},x_{3})\mu(x_{1})-\mu(x_{1}) S_{\rho,\mu}(x_{2},x_{3}),
   \end{eqnarray*}
   where $R_{\rho,\mu},S_{\rho,\mu}: A\otimes A\rightarrow \gl(V)$ are defined by
   \begin{eqnarray}
  \label{eq:repH 1}R_{\rho,\mu}(x_{1},x_{2})&=&\rho(x_{1})\mu(x_{2})-\varepsilon(x_{1},x_{2})\mu(x_{2})\rho(x_{1})-\mu([x_{1},x_{2}]),\\
                  S_{\rho,\mu}(x_{1},x_{2})&=&\mu(x_{1})\rho(x_{2})+\varepsilon(x_{1},x_{2})\mu(x_{2})\rho(x_{1})-\rho(x_{1}\cdot x_{2})
   \end{eqnarray}
  for all homogeneous elements $x_{1},x_{2},x_{3}\in A$.
\end{defi}

In the rest of the paper, we  sometimes denote $R_{\rho,\mu}, S_{\rho,\mu}$  by $R, S$ respectvely if there is no risk of confusion.

\begin{ex}
  Let $(A,\cdot,[ , ],\varepsilon)$ be an $F$-manifold color algebra. Then $(A,\ad,\huaL)$ is a representation of $A$,
  where $\ad: A\longrightarrow \gl(A)$ defined by $\ad_{x}y=[x,y]$, and multiplication operator $\huaL:A\longrightarrow \gl(A)$ are defined  by $\huaL_{x}y=x\cdot y$ for all homogeneous elements $x, y\in A$.
\end{ex}

\begin{proof}
It is known that $(V,\ad)$ is a  representation of the Lie color algebra $(A,[ , ], \varepsilon)$ and $(V,\huaL)$ is a representation of the
$\varepsilon$-commutative associative algebra $(A,\cdot, \varepsilon)$.

Note that, for all homogeneous elements $x, y, z, w\in A$, we have
\begin{eqnarray*}
R(x, y)z&=& (ad_{x}\huaL_y-\varepsilon(x, y)\huaL_y ad_{x}-\huaL_{[x, y]})\cdot z\\
&=& [x, y\cdot z]-\varepsilon(x, y)y\cdot [x, y]-[x, y]\cdot z\\
&=& P_{x}(y, z).
\end{eqnarray*}
Thus
\begin{eqnarray*}
P_{x\cdot y} (z,w)=x\cdot P_y(z, w)+\varepsilon(x, y)y\cdot P_x(z, w)
\end{eqnarray*}
implies the equation
\begin{eqnarray*}
R(x\cdot y, z)w=\huaL_{x}R(y,z)w+\varepsilon(x, y)\huaL_{y}R(x,z)w.
\end{eqnarray*}
On the other hand
\begin{eqnarray*}
&&S(y, z)w=(\huaL_{y}ad_{z}+\varepsilon(y, z)\huaL_{z} ad_{y}-ad_{y\cdot z})w\\
&=& y\cdot [z, w]+\varepsilon(y, z)z\cdot [y, w]-[y\cdot z, w]\\
&=&-\varepsilon(z,w)y\cdot [w, z]-\varepsilon(y, w)\varepsilon(z, w)[w, y]\cdot z+\varepsilon(y+z, w)[w, y\cdot z]\\
&=&\varepsilon(y+z,w)([w, y\cdot z]- [w, y]\cdot z-\varepsilon(w, y)y\cdot [w, z])\\
&=&\varepsilon(y+z, w)P_w (y, z).
\end{eqnarray*}
Thus
\begin{eqnarray*}
&&\varepsilon(x,y+z)S(y, z)\huaL_{x}w-\huaL_xS(y, z)w\\
&=&\varepsilon(x,y+z)S(y, z)(x\cdot w)-x\cdot S(y, z)w\\
&=&\varepsilon(x,y+z)\varepsilon(y+z,x+w)P_{x\cdot w}(y, z)-\varepsilon(y+z, w)x\cdot P_w(y, z)\\
&=&\varepsilon(y+z, w)\{P_{x\cdot w}(y, z)-x\cdot P_w(y, z)\}\\
&=&\varepsilon(y+z, w)\varepsilon(x, w)w\cdot P_x(y, z)\\
&=&\varepsilon(x+y+z, w)w\cdot P_x(y, z)\\
&=&P_x(y, z)w.
\end{eqnarray*}
Hence, $(A,\ad,\huaL)$ is a representation of $A$.
\end{proof}

Let $(A,\cdot,[ , ], \varepsilon)$  be an $F$-manifold color algebra and $(V, \rho,\mu)$ a representation of $A$.  Note that $A\oplus V$ is  a G-graded vector space.  In the following, if we write $x+v\in A\oplus V$ as a homogeneous  element,
it means that $x,v$ are of the same degree as $x+v$.  For any homogeneous elements $x_1+v_1,x_2+v_2\in A\oplus V$, define
$$
  [x_1+v_1,x_2+v_2]_\rho=[x_1,x_2]+\rho(x_1)v_2-\varepsilon(x_{1},x_{2})\rho(x_2)v_1.
$$
It is well known that $(A\oplus V, [ , ]_\rho,\varepsilon)$ is a Lie color algebra.
Moreover, if one defines
$$
  (x_1+v_1)\cdot_{\mu}(x_2+v_2)=x_1\cdot x_2+\mu(x_1)v_2+\varepsilon(x_{1},x_{2})\mu(x_2)v_1,$$ then
 $(A\oplus V,\cdot_\mu, \varepsilon)$ is an $\varepsilon$-commutative associative algebra. In fact, we have
\begin{pro}\label{pro:semi-direct}
 With the above notations, let $(A,\cdot,[ , ], \varepsilon)$  be an $F$-manifold color algebra and $(V, \rho,\mu)$ a representation of $A$. Then
 $(A\oplus V,\cdot_{\mu},[ , ]_\rho,\varepsilon)$ is an $F$-manifold color algebra.
 \end{pro}
\begin{proof}
We only need to prove that the  color  Hertling-Manin relation holds.

For any homogeneous elements $x_1+v_1,x_2+v_2,x_3+v_3\in A\oplus V$, we have
\begin{eqnarray*}
&& P_{x_1+v_1}(x_2+v_2,x_3+v_3)\\
&=&[x_1+v_1,(x_2+v_2)\cdot_{\mu} (x_3+v_3)]_\rho-[x_1+v_1,x_2+v_2]_\rho \cdot_{\mu} (x_3+v_3)\\
&&-\varepsilon(x_1,x_2)( x_2+v_2)\cdot_{\mu} [x_1+v_1,x_3+v_3]_\rho\\
&=& [x_1, x_2\cdot x_3]+\rho(x_1)\{\mu(x_2)v_3+\varepsilon(x_2,x_3)\mu(x_3)v_2\}-\varepsilon(x_1,x_2+ x_3)\rho (x_2\cdot x_3)v_1-I-II.\\
\end{eqnarray*}
where
\begin{eqnarray*}
I&=&\{[x_1,x_2]+\rho(x_1)v_3-\varepsilon(x_1,x_2)\rho(x_2)v_1\}\cdot_{\mu} (x_3+v_3)\\
&=&[x_1,x_2]\cdot x_3+\mu([x_1,x_2])v_3+\varepsilon(x_1+x_2,x_3)\mu(x_3)\{\rho(x_1)v_2 -\varepsilon(x_1,x_2)\rho(x_2)v_1 \},\\
\end{eqnarray*}
and
\begin{eqnarray*}
II&=&\varepsilon(x_1,x_2)( x_2+v_2)\cdot_{\mu} \{[x_1,x_3]+\rho(x_1)v_3- \varepsilon(x_1,x_3)\rho(x_3)v_1\}\\
&=&\varepsilon(x_1,x_2)\{x_2\cdot [x_1, x_3]+\mu(x_2)(\rho(x_1)v_3-\varepsilon(x_1,x_3)\rho(x_3)v_1)\\
&&+\varepsilon(x_2,x_1+ x_3)\mu ([x_1,x_3])v_2\}.
\end{eqnarray*}
Thus
\begin{eqnarray*}
&&P_{x_1+v_1}(x_2+v_2,x_3+v_3)\\
&=&
P_{x_1}(x_2,x_3)+\{\rho(x_1)\mu(x_2)-\mu([x_1,x_2])-\varepsilon(x_1,x_2)\mu(x_2)\rho(x_1)\}v_3\\
&&+\{\varepsilon(x_2,x_3)\rho(x_1)\mu(x_3)-\varepsilon(x_1+x_2,x_3)\mu(x_3)\rho(x_1)\\
&&-\varepsilon(x_1,x_2)\varepsilon(x_2,x_1+x_3)\mu([x_1,x_3])\}v_2
+\{-\varepsilon(x_1,x_2+x_3)\rho(x_2\cdot x_3)\\
&&+\varepsilon(x_1+x_2,x_3)\varepsilon(x_1,x_2)\mu(x_3)\rho(x_2)+\varepsilon(x_1,x_2)\varepsilon(x_1,x_3)\mu(x_2)\rho(x_3)\}v_1\\
&=&P_{x_1}(x_2,x_3)+R(x_1,x_2)v_3 +\varepsilon(x_2,x_3)R(x_1,x_3)v_2+\varepsilon(x_1,x_2+x_3)S(x_2,x_3)v_1.
\end{eqnarray*}
Hence, for any homogeneous element $x_4+v_4\in A\oplus V$, we have
\begin{eqnarray*}
&& P_{(x_1+v_1)\cdot_{\mu}(x_2+v_2)}(x_3+v_3,x_4+v_4)\\
&=&P_{x_1\cdot x_2+\mu(x_1)v_2+\varepsilon(x_1,x_2)\mu(x_2)v_1}(x_3+v_3,x_4+v_4)\\
&=&P_{x_1\cdot x_2}(x_3,x_4)+R(x_1\cdot x_2,x_3)v_4+\varepsilon(x_3,x_4)R(x_1\cdot x_2,x_4)v_3\\
&&+\varepsilon(x_1+x_2,x_3+x_4)S(x_3,x_4)(\mu(x_1)v_2+\varepsilon(x_1,x_2)\mu(x_2)v_1).
\end{eqnarray*}
On the other hand
\begin{eqnarray*}
&&(x_1+v_1)\cdot_{\mu}P_{x_2+v_2}(x_3+v_3,x_4+v_4)\\
&=&(x_1+v_1)\cdot_{\mu}\{P_{x_2}(x_3,x_4)+R(x_2,x_3)v_4+\varepsilon(x_3,x_4)R(x_2,x_4)v_3+\varepsilon(x_2,x_3+x_4)S(x_3,x_4)v_2\}\\
&=&x_1\cdot P_{x_2}(x_3,x_4)+\mu(x_1)\{R(x_2,x_3)v_4+\varepsilon(x_3,x_4)R(x_2,x_4)v_3+\varepsilon(x_2,x_3+x_4)S(x_3,x_4)v_2\}\\
&&+\varepsilon(x_1,x_2+x_3+x_4)\mu(P_{x_2}(x_3,x_4))v_1,
\end{eqnarray*}
and
\begin{eqnarray*}
&&\varepsilon(x_1,x_2)(x_2+v_2)\cdot_{\mu} P_{x_1+v_1}(x_3+v_3,x_4+v_4)\\
&=&\varepsilon(x_1,x_2)\{x_2\cdot P_{x_1}(x_3,x_4)+\mu(x_2)\{R(x_1,x_3)v_4+\varepsilon(x_3,x_4)R(x_1,x_4)v_3\\
&&+\varepsilon(x_1,x_3+x_4)S(x_3,x_4)v_1\}+\varepsilon(x_2,x_1+x_3+x_4)\mu(P_{x_1}(x_3,x_4))v_2\}.
\end{eqnarray*}
Thus
\begin{eqnarray*}
&& (x_1+v_1)\cdot_{\mu}P_{x_2+v_2}(x_3+v_3,x_4+v_4)+\varepsilon(x_1,x_2)(x_2+v_2)\cdot_{\mu}P_{x_1+v_1}(x_3+v_3,x_4+v_4)\\
&=& x_1\cdot P_{x_2}(x_3,x_4)+\varepsilon(x_1,x_2)x_2\cdot P_{x_1}(x_3,x_4)\\
&&+ \{\mu(x_1)R(x_2,x_3)+\varepsilon(x_1,x_2)\mu(x_2)(R(x_1,x_3))\}v_4\\
&&+ \{\varepsilon(x_3,x_4)\mu(x_1)R(x_2,x_4)+\varepsilon(x_1,x_2)\varepsilon(x_3,x_4)\mu(x_2)R(x_1,x_4)\}v_3\\
&&+ \{\varepsilon(x_2,x_3+x_4)\mu(x_1)S(x_3,x_4)+\varepsilon(x_1,x_2)\varepsilon(x_2,x_1+x_3+x_4)\mu(P_{x_1}(x_3,x_4))\}v_2\\
&&+ \varepsilon(x_1,x_2+x_3+x_4)\{\mu(x_2)S(x_3,x_4)+\mu(P_{x_2}(x_3,x_4))\}v_1\\
&=& P_{(x_1+v_1)\cdot_{\mu}(x_2+v_2)}(x_3+v_3,x_4+v_4).
\end{eqnarray*}
By the definition of an $F$-manifold color algebra, the conclusion follows immediately.
\end{proof}

Let $(V, \rho,\mu)$ be a representation of an $F$-manifold algebra, then the triple $(V^*, \rho^*,-\mu^*)$ is not necessarily  a representation of this $F$-manifold algebra \cite{LSB}. In the following, we prove the version of an $F$-manifold color algebra.

\begin{pro}
  Let $(A,\cdot,[ , ],\varepsilon)$ be an $F$-manifold color algebra,  $(V,\rho)$  a  representation of the Lie color algebra
  $(A,[ , ], \varepsilon)$, and $(V,\mu)$  a representation of the $\varepsilon$-commutative associative algebra $(A,\cdot, \varepsilon)$, such that, for any homogeneous elements $x,y,z\in A$,
 \begin{eqnarray*}
     \label{eq:corep 1}R_{\rho,\mu}(x\cdot y,z)=\varepsilon(x,y+z) R_{\rho,\mu}(y,z)\mu(x)+ \varepsilon(y,z)R_{\rho,\mu}(x,z) \mu(y),\\
     \label{eq:corep 2} \mu(P_x(y,z))=-\varepsilon(x,y+z)T_{\rho,\mu}(y,z)\mu(x)+\mu(x) T_{\rho,\mu}(y,z),
   \end{eqnarray*}
   where $R_{\rho,\mu}$ is given by \eqref{eq:repH 1} and $T_{\rho,\mu}: A\otimes A\rightarrow \gl(V)$ is defined by
   \begin{eqnarray*}
    \label{eq:repH 3} T_{\rho,\mu}(x,y)&=&-\varepsilon(x,y)\rho(y)\mu(x)-\rho(x)\mu(y)+\rho(x\cdot y),
   \end{eqnarray*}
   then $(V^*,\rho^*,-\mu^*)$ is  a representation of $A$.
\end{pro}
\begin{proof}

Assume that $x,y,z\in A,v\in V,\alpha\in V^*$ are all homogeneous elements.  First, we claim the following two identities:
\begin{eqnarray*}
\langle R_{\rho^*,-\mu^*}(x,y)(\alpha),v\rangle&=&\langle\alpha,\varepsilon(x+y,\alpha)R_{\rho,\mu}(x,y)v\rangle;\\
\langle S_{\rho^*,-\mu^*}(x,y)(\alpha),v\rangle&=&\langle\alpha,\varepsilon(x+y,\alpha)T_{\rho,\mu}(x,y)v\rangle.
\end{eqnarray*}
The claims follow from  some direct calculations, respectively:
\begin{eqnarray*}
&&\langle R_{\rho^*,-\mu^*}(x,y)(\alpha),v\rangle\\
&=&\langle(-\rho^*(x)\mu^*(y)+\varepsilon(x,y)\mu^*(y)\rho^*(x)+\mu^*([x,y])\alpha,v\rangle\\
&=&\varepsilon(x,y+\alpha)\langle\mu^*(y)\alpha,\rho(x)v\rangle-\varepsilon(x,y)\varepsilon(y,x+\alpha)\langle(\rho^*(x)\alpha,\mu(y)v\rangle\\
&&-\varepsilon(x+y,\alpha)\langle\alpha,\mu([x,y])v\rangle\\
&=&-\varepsilon(x,y)\varepsilon(x+y,\alpha)\langle\alpha,\mu(y)\rho(x)v\rangle
+\varepsilon(y, \alpha)\varepsilon(x,\alpha)\langle\alpha,\rho(x)\mu(y)v\rangle\\
&&-\varepsilon(x+y,\alpha)\langle\alpha,\mu([x,y])v\rangle\\
&=&\langle\alpha,\varepsilon(x+y,\alpha)\{-\varepsilon(x, y)\mu(y)\rho(x)+\rho(x)\mu(y)-\mu([x,y])\}v\rangle\\
&=&\langle\alpha,\varepsilon(x+y,\alpha)R_{\rho,\mu}(x,y)v\rangle,
\end{eqnarray*}
and
\begin{eqnarray*}
&&\langle S_{\rho^*,-\mu^*}(x,y)(\alpha),v\rangle\\
&=&\langle\{-\mu^*(x)\rho^*(y)-\varepsilon(x,y)\mu^*(y)\rho^*(x)-\rho^*(x\cdot y)\}\alpha,v\rangle\\
&=&-\varepsilon(x,y+\alpha)\varepsilon(y,\alpha)\langle\alpha,\rho(y)\mu(x)v\rangle
-\varepsilon(y,\alpha)\varepsilon(x,\alpha)\langle\alpha,\rho(x)\mu(y)v\rangle \\
&&+\varepsilon(x+y,\alpha)\langle\alpha,\rho(x\cdot y)v\rangle\\
&=&\langle\alpha,\varepsilon(x+y,\alpha)\{-\varepsilon(x,y)\rho(y)\mu(x)-\rho(x)\mu(y)+\rho(x\cdot y)\}v\rangle\\
&=&\langle\alpha,\varepsilon(x+y,\alpha)T_{\rho,\mu}(x,y)v\rangle.
\end{eqnarray*}

With the above identities, we have
\begin{eqnarray*}
 &&\langle\{ R_{\rho^*,-\mu^*}(x\cdot y,z)+\mu^*(x) R_{\rho^*,-\mu^*}(y,z)+\varepsilon(x,y)\mu^*(y) R_{\rho^*,-\mu^*}(x,z)\}\alpha,v\rangle\\
 &=&\langle\alpha,\varepsilon(x+y+z,\alpha)R_{\rho,\mu}(x\cdot y,z)v\rangle
 - \varepsilon(x,y+z+\alpha)\varepsilon(y+ z, \alpha)\langle \alpha, R_{\rho,\mu}(y,z)\mu(x)v\rangle\\
 &\quad&-\varepsilon(x+z,\alpha)\varepsilon(y,z+\alpha)\langle \alpha, R_{\rho,\mu}(x,z)\mu(y)v\rangle \\
 &=&\varepsilon(x+y+z,\alpha)\langle \alpha ,  \{R_{\rho,\mu}(x\cdot y,z)-\varepsilon(x,y+z)R_{\rho,\mu}(y,z)\mu(x)
 -\varepsilon(y,z) R_{\rho,\mu}(x,z)\mu(y)\}v  \rangle \\
&=&0.
\end{eqnarray*}

And
\begin{eqnarray*}
 &&\langle \{-\mu^*(P_x(y,z))+\varepsilon(x,y+z)S_{\rho^*,-\mu^*}(y,z)\mu^*(x)-\mu^*(x) S_{\rho^*,-\mu^*}(y,z)\}\alpha,v\rangle\\
 &=&\varepsilon(x+y+z,\alpha) \langle \alpha,\mu(P_x(y,z))v\rangle
 + \varepsilon(x,y+z)\varepsilon(y+z,x+\alpha)\langle \mu^*(x)\alpha,  T_{\rho,\mu}(y,z)v\rangle\\
 &\quad&  + \varepsilon(x,y+z+\alpha)\langle S_{\rho^*,-\mu^*}(y,z)\alpha , \mu(x) v \rangle\\
 &=&\varepsilon(x+y+z,\alpha) \langle \alpha,\mu(P_x(y,z))v\rangle
  -\varepsilon(y+z,\alpha)\varepsilon(x,\alpha)\langle \alpha,  \mu(x)T_{\rho,\mu}(y,z)v\rangle\\
 &\quad&  + \varepsilon(x,y+z+\alpha)\varepsilon(y+z,\alpha)\langle\alpha ,  T_{\rho,\mu}(y,z)\mu(x) v \rangle\\
 &=&\varepsilon(x+y+z,\alpha)\langle \alpha,\{\mu(P_x(y,z))
  -\mu(x)T_{\rho,\mu}(y,z) + \varepsilon(x,y+z)T_{\rho,\mu}(y,z)\mu(x)\}v \rangle\\
 &=&0.
\end{eqnarray*}
Thus, the conclusion follows immediately from the definition of representations and the hypothesis.
\end{proof}

\begin{defi}
  A  coherence $F$-manifold color algebra is an $F$-manifold color algebra such that,
for all homogeneous elements $x,y,z,w\in A$, the following hold
  \begin{eqnarray*}
   P_{x\cdot y}(z,w)&=&\varepsilon(x,y+z)P_y(z,x\cdot w)+\varepsilon(y,z)P_x(z,y\cdot w),\\
   P_x(y,z) w&=&-\varepsilon(x,y+z) T(y,z)(x\cdot w)+xT(y,z)(w),
  \end{eqnarray*}
 where
  $$ T(y,z)(w)=-\varepsilon(y,z)[z, y\cdot w]-[y, z\cdot w]+[y\cdot z, w].$$
\end{defi}

\begin{pro}
Let $(A,\cdot,[ , ], \varepsilon)$ be an $F$-manifold color algebra, and $\frak B$
 a nondegenerate symmetric bilinear form on $A$ satisfying
\begin{equation*}
\frak B(x\cdot y,z)=\frak B(x,y\cdot z),\;\;\frak
B([x,y],z)=\frak B(x,[y,z]),
\end{equation*}
for all homogeneous elements $x,y,z\in A$,
then $(A,\cdot,[ , ], \varepsilon)$ is a coherence $F$-manifold color algebra.
\end{pro}
\begin{proof}
First, we prove that
$$\frkB(P_x(y,z),w)=\varepsilon(x+y,z)\frkB(z, P_x(y,w))$$
for all homogeneous elements $x,y,z,w\in A$.

In fact, we have
\begin{eqnarray*}
 &&\frkB(P_x(y,z),w)\\
 &=& \frkB ([x,y\cdot z]-[x,y]\cdot z-\varepsilon(x,y) y\cdot [x,z], w)\\
 &=&-\varepsilon(x, y+z)\frkB ([y\cdot z, x], w)-\varepsilon(x+y, z)\frkB (z, [x,y]\cdot w)
 -\varepsilon(x,y)\varepsilon(y,x+z)\frkB ([x,z], y\cdot w)\\
 &=&-\varepsilon(x, y+z)\frkB (y\cdot z, [x, w])-\varepsilon(x+y, z)\frkB (z, [x,y]\cdot w)
 +\varepsilon(y,z)\varepsilon(x,z)\frkB (z, [x, y\cdot w])\\
  &=&-\varepsilon(x, y+z)\varepsilon(y, z)\frkB (z,y\cdot [x, w])-\varepsilon(x+y, z)\frkB (z, [x,y]\cdot w)
 +\varepsilon(x+y,z)\frkB (z, [x, y\cdot w])\\
&=& \varepsilon(x+y,z) \frkB (z, -\varepsilon(x, y)y\cdot [x, w]-[x,y]\cdot w+[x, y\cdot w])\\
&=& \varepsilon(x+y,z) \frkB (z, P_x(y, w)).
\end{eqnarray*}
By the above relation, for all homogeneous elements $x,y,z,w_1,w_2\in A$, we have
\begin{eqnarray*}
&&\frkB(P_{x\cdot y}(z,w_1)-\varepsilon(x, y+z) P_y(z,x\cdot w_1)-\varepsilon(y,z)P_x(z,y\cdot w_1),w_2)\\
&=&\varepsilon(x+y+z,w_1) \frkB(w_1, P_{x\cdot y}(z,w_2))-\varepsilon(x,y+z)\varepsilon(y+ z,x+ w_1) \frkB(x\cdot w_1, P_y(z,w_2))\\
&\quad& -\varepsilon(y, z)\varepsilon(x\cdot z, y\cdot w_1) \frkB(y\cdot w_1, P_x(z,w_2))\\
&=&\varepsilon(x+y+z,w_1) \frkB(w_1, P_{x\cdot y}(z,w_2))-\varepsilon(x,y+z)\varepsilon(y\cdot z,x\cdot w_1) \varepsilon(x, w_1)
\frkB(w_1, x\cdot P_y(z,w_2))\\
&\quad& -\varepsilon(y, z)\varepsilon(x\cdot z, y\cdot w_1) \varepsilon(y, w_1)\frkB(w_1, y\cdot P_x(z,w_2))\\
&=&\varepsilon(x+y+z,w_1) \frkB(w_1, P_{x\cdot y}(z,w_2))-\varepsilon(x+y+z,w_1)\frkB(w_1, x\cdot P_y(z,w_2))\\
&\quad& -\varepsilon(y, z)\varepsilon(x+z, y) \varepsilon(x+y+z, w_1)\frkB(w_1, y\cdot P_x(z,w_2))\\
&=&\varepsilon(x+y+z,w_1) \frkB(w_1, P_{x\cdot y}(z,w_2))-\varepsilon(x+y+z,w_1)\frkB(w_1, x\cdot P_y(z,w_2))\\
&\quad& -\varepsilon(x, y)\varepsilon(x+y+z, w_1)\frkB(w_1, y\cdot P_x(z,w_2))\\
&=&\varepsilon(x+y+z,w_1) \frkB(w_1, P_{x\cdot y}(z,w_2)-x\cdot P_y(z,w_2) -\varepsilon(x, y)y\cdot P_x(z,w_2))\\
&=&0.
\end{eqnarray*}
We claim the following identity:
$$\frkB(T(y,z)(x\cdot w_1), w_2)=\varepsilon(y+z, w_1+w_2)\frkB(w_1, P_{w_2}(y,z)).$$
In fact, we have
\begin{eqnarray*}
&&\frkB(T(y,z)(x\cdot w_1), w_2)\\
&=&\frkB(-\varepsilon(y, z)[z,y\cdot w_1]-[y, z\cdot w_1]+[y\cdot z, w_1], w_2)\\
&=&\varepsilon(y, z)\varepsilon(z, y+w_1)\frkB(y\cdot w_1, [z, w_2])+\varepsilon(y, z+ w_1)\frkB(z\cdot w_1, [y, w_2])
-\varepsilon(y+z, w_1)\frkB(w_1, [yz, w_2])\\
&=&\varepsilon(z, w_1)\varepsilon(y, w_1)\frkB(w_1, y\cdot [z, w_2])+\varepsilon(y, z+ w_1)\varepsilon(z, w_1)\frkB(w_1, z\cdot [y, w_2])\\
&&-\varepsilon(y+z, w_1)\frkB(w_1, [y\cdot z, w_2])\\
&=&\varepsilon(y+z, w_1)\frkB(w_1, y\cdot [z, w_2])+\varepsilon(y+z, w_1)\varepsilon(y, z)\frkB(w_1, z\cdot [y, w_2])
-\varepsilon(y+z, w_1)\frkB(w_1, [y\cdot z, w_2])\\
&=&\varepsilon(y+z, w_1)\frkB(w_1, y\cdot [z, w_2]+\varepsilon(y, z)z\cdot [y, w_2]-[y\cdot z, w_2])\\
&=&\varepsilon(y+z, w_1)\frkB(w_1, \varepsilon(y+z, w_2)P_{w_2}(y,z))\\
&=&\varepsilon(y+z, w_1+w_2)\frkB(w_1, P_{w_2}(y,z)).\\
\end{eqnarray*}
With the above identity, we have
\begin{eqnarray*}
&&\frkB(P_x(y,z)\cdot w_1+\varepsilon(x, y+z)T(y,z)(x\cdot w_1)-x\cdot T(y,z)(w_1),w_2)\\
&=&\varepsilon(x+y+z, w_1)\frkB(w_1, P_{x}(y,z)w_2)+\varepsilon(x, y+z)\varepsilon(y+z, x+w_1+w_2)\frkB(x\cdot w_1, P_{w_2}(y, z))\\
&\quad&  -\varepsilon(x, y+z+w_1)\frkB(T(y, z)w_1, x\cdot w_2)\\
&=&\varepsilon(x+y+z, w_1)\frkB(w_1, P_{x}(y, z)w_2)+\varepsilon(x, w_1)\varepsilon(y+z, w_1+w_2)\frkB(w_1, x\cdot P_{w_2}(y, z))\\
&\quad&  -\varepsilon(x, y+z+w_1)\varepsilon(y+z, x+w_1+w_2)\frkB(w_1, P_{x\cdot w_2}(y, z))\\
&=&\varepsilon(x+y+z, w_1)\frkB(w_1, P_{x}(y,z)w_2)+\varepsilon(x, w_1)\varepsilon(y+z, w_1+w_2)\frkB(w_1, x\cdot P_{w_2}(y, z))\\
&\quad&  -\varepsilon(x, w_1)\varepsilon(y+z, w_1+w_2)\frkB(w_1, P_{x\cdot w_2}(y, z))\\
&=&\varepsilon(x+y+z, w_1)\frkB(w_1, P_{x}(y, z)w_2+\varepsilon(y+z, w_2)x\cdot P_{w_2}(y, z)-\varepsilon(y+z, w_2)P_{x\cdot w_2}(y, z))\\
&=&\varepsilon(x+y+z, w_1)\frkB(w_1, P_{x}(y, z)w_2+\varepsilon(y+z, w_2)x\cdot P_{w_2}(y, z)\\
&&-(P_{x}(y, z)w_2+\varepsilon(y+z, w_2)x\cdot P_{w_2}(y, z)))\\
&=&0.
\end{eqnarray*}

By  the assumptions that $A$ is an $F$-manifold color algebra and $\frkB$ is nondegenerate, the conclusion is obtained.
\end{proof}

\section{Pre-$F$-manifold color algebras}
In this section, we introduce the notion of a  pre-$F$-manifold color algebra,
and then construct an $F$-manifold color algebra from a pre-$F$-manifold color algebra.

\begin{defi}
A triple $(A,\diamond, \varepsilon)$ is called a  Zinbiel color algebra if
$A$ is a G-graded vector space and  $\diamond:A\otimes A\longrightarrow A$ is
a bilinear multiplication satisfying
\begin{enumerate}
\item   $A_{a}\diamond A_{b}\subseteq A_{a+b},$
\item   $ x\diamond(y\diamond z)=(x\diamond y)\diamond z+\varepsilon(x, y)(y\diamond x)\diamond z,$
\end{enumerate}
for all $x\in A_{a}, y\in A_{b}, z\in A_{c}$, and $a, b, c\in G$.
\end{defi}

Let $(A,\diamond, \varepsilon)$ be a Zinbiel color algebra.  For any homogeneous elements $x, y\in A$, define $$x\cdot y=x\diamond y+\varepsilon(x, y)y\diamond x,$$
and  define $\frkL: A\longrightarrow\gl(A)$ by
\begin{equation}
\frkL_{ x}y=x\diamond y.
\end{equation}

\begin{lem}\label{lem4.2}
With the above notations, we have that $(A,\cdot, \varepsilon)$ is an $\varepsilon$-commutative associative algebra, and $(A,\frkL)$ is its representation.
\end{lem}

\begin{proof}
First, for any homogeneous elements $x, y\in A$, we have
\begin{eqnarray*}
&&\varepsilon(x, y)y \cdot x=\varepsilon(x, y)(y\diamond x+\varepsilon(y, x)x\diamond y  )= x\diamond y+\varepsilon(x, y)y\diamond x=x\cdot y.
\end{eqnarray*}
Thus, the $\varepsilon$-commutativity follows.

Secondly, for any homogeneous elements $x, y, z\in A$, we have
\begin{eqnarray*}
&&(x\cdot y)\cdot z=(x\diamond y+\varepsilon(x, y)y\diamond x  )\cdot z\\
&=& (x\diamond y)\diamond z+\varepsilon(x+y, z)z\diamond (x\diamond y)
+\varepsilon(x, y)\{(y\diamond x)\diamond z+\varepsilon(x+y, z)z\diamond (y\diamond x) \}\\
&=& (x\diamond y)\diamond z+\varepsilon(x, y)(y\diamond x)\diamond z+\varepsilon(x, y)\varepsilon(x+y, z)z\diamond (y\diamond x)
+\varepsilon(x+y, z)z\diamond (x\diamond y)\\
&=& (x\diamond y)\diamond z+\varepsilon(x, y)(y\diamond x)\diamond z+
\varepsilon(x, y)\varepsilon(x+y, z)\{(z\diamond y)\diamond x +\varepsilon(z, y)(y\diamond z)\diamond x\}\\
&&+\varepsilon(x+y, z)\{(z\diamond x)\diamond y +\varepsilon(z, x)(x\diamond z)\diamond y\}\\
&=& \{(x\diamond y)\diamond z+\varepsilon(x, y)(y\diamond x)\diamond z\}+
\{\varepsilon(x, y+z)(y\diamond z)\diamond x+\varepsilon(x, y+z)\varepsilon(y, z)(z\diamond y)\diamond x\}\\
&&+\varepsilon(y, z)\{(x\diamond z)\diamond y+\varepsilon(x, z)(z\diamond x)\diamond y\}\\
&=& x\diamond (y\diamond z)+\varepsilon(x, y+z)(y\diamond z)\diamond x+\varepsilon(y, z)\varepsilon(x, y+z)(z\diamond y)\diamond x
+\varepsilon(y, z)x\diamond (z\diamond y),
\end{eqnarray*}
and
\begin{eqnarray*}
&&x\cdot (y\cdot z)=x\cdot (y\diamond z+\varepsilon(y, z)z\diamond y  )\\
&=& x\diamond (y\diamond z)+\varepsilon(x, y+z)(y\diamond z)\diamond x
+\varepsilon(y, z)\{x\diamond (z\diamond y)+\varepsilon(x, y+z)(z\diamond y)\diamond x \}\\
&=& x\diamond (y\diamond z)+\varepsilon(x, y+z)(y\diamond z)\diamond x+\varepsilon(y, z)\varepsilon(x, y+z)(z\diamond y)\diamond x
+\varepsilon(y, z)x\diamond (z\diamond y).\\
\end{eqnarray*}
Thus, the associativity follows.

From the definition of $\frkL $, it follows that
$$
\frkL_{ xy}z=(x\cdot y) \diamond z=(x\diamond y+\varepsilon(x, y)(y\diamond x))\diamond z
=x\diamond  (y\diamond  z)=\frkL_{x} \frkL_{y}(z).
$$
Hence, $(A,\frkL)$ is a representation of the $\varepsilon$-commutative associative algebra $(A,\cdot, \varepsilon)$.
\end{proof}

\begin{defi}\label{pre}
  $(A,\diamond,\ast,\varepsilon)$ is called a  pre-$F$-manifold color algebra if $(A,\diamond, \varepsilon)$ is a  Zinbiel color algebra and
  $(A,\ast, \varepsilon)$ is a pre-Lie color algebra, such that,  for all homogeneous elements $x,y,z,w\in A$,
  $$F_1(x\cdot y,z,w)=x\diamond F_1(y,z,w)+\varepsilon(x,y)y\diamond F_1(x,z,w),$$
  \begin{eqnarray*}
    &&(F_1(x,y,z)+\varepsilon(y, z)F_1(x,z,y)+\varepsilon(x, y+z)F_2(y,z,x))\diamond w\\
    &=&\varepsilon(x, y+z)F_2(y,z,x\diamond w)-x\diamond F_2(y,z,w).
  \end{eqnarray*}
where $F_1,F_2:\otimes^3 A\longrightarrow A$ are defined by
  \begin{eqnarray*}
    F_1(x,y,z)&=&x\ast(y\diamond z)-\varepsilon(x,y)y\diamond(x\ast z)-[x,y]\diamond z,\\
    F_2(x,y,z)&=&x\diamond(y\ast z)+\varepsilon(x,y)y\diamond(x\ast z)-(x\cdot y)\ast z,
  \end{eqnarray*}
and the operation $\cdot$ and bracket $[ , ]$ are defined by
\begin{equation*}\label{eq:pHM-operations}
  x\cdot y=x\diamond y+\varepsilon(x, y)y\diamond x,\quad [x,y]=x\ast y-\varepsilon(x,y)y\ast x.
\end{equation*}
\end{defi}

It is well known that $(A,[ , ], \varepsilon)$ is a Lie color algebra, and $(A, L)$ is a representation of the  Lie color algebra $(A, [ , ], \varepsilon)$ where
$L: A\longrightarrow\gl(A)$ is defined by
$$L_{ x}y=x\ast y,$$ for any homogeneous elements $x,y\in A$.

\begin{thm}
  Let $(A,\diamond,\ast, \varepsilon)$ be a pre-$F$-manifold color algebra. Then
   \begin{itemize}
\item[$\rm(1)$]
   $(A,\cdot,[ , ],\varepsilon)$ is an $F$-manifold color algebra, where the operation $\cdot$ and bracket $[ , ]$ are defined in Definition \ref{pre}.
\item[$\rm(2)$]$(A;L,\frkL)$ is a representation of $(A,\cdot,[ , ],\varepsilon)$, where $L, \frkL:A\longrightarrow\gl(A)$ are defined by
$$L_{ x}y=x\ast y,\ \ \ \frkL_{ x}y=x\diamond y,$$
for any homogeneous elements $x,y\in A$.

\end{itemize}
\end{thm}
\begin{proof}
(1)\  By Lemma \ref{lem4.2}, $(A,\cdot, \varepsilon)$ is an $\varepsilon$-commutative associative algebra.
It is known that $(A,[ , ], \varepsilon)$ is a Lie color algebra. Thus we only need to prove that the  color  Hertling-Manin relation holds.

Assume that $x,y,z,w\in A$ are all homogeneous elements.
We claim the following identity:
\begin{equation}\label{equ}
  P_x(y,z)=F_1(x,y,z)+\varepsilon(y,z)F_1(x,z,y)+\varepsilon(x, y+z)F_2(y,z,x).
\end{equation}
In fact, we have
  \begin{eqnarray*}
    &&P_x(y,z)=[x,y\cdot z]-[x,y]\cdot z-\varepsilon(x,y)y\cdot [x,z]\\
    &=&x\ast (y\cdot z)-\varepsilon(x,y+z)(y\cdot z)\ast x-[x,y]\diamond z-\varepsilon(x+y,z)z\diamond [x,y]\\
     &&-\varepsilon(x, y)\{y\diamond [x,z]+\varepsilon(y,x+z)[x,z]\diamond y\}\\
     &=&x\ast (y\diamond z)-\varepsilon(x,y)y\diamond (x\ast z)-[x,y]\diamond z+\varepsilon(y,z)\{
     x\ast (z\diamond y)-\varepsilon(x,z)z\diamond (x\ast y)-[x,z]\diamond y\}\\
     &&+\varepsilon(x, y+z)\{y\diamond (z\ast x)+\varepsilon(y,z)z\diamond (y\ast x)-(yz)\ast x\}\\
     &=&F_1(x,y,z)+\varepsilon(y,z)F_1(x,z,y)+\varepsilon(x, y+z)F_2(y,z,x).
  \end{eqnarray*}
With the above identity,  we have
  \begin{eqnarray*}
    &&P_{x\cdot y}(z,w)-x\cdot P_{y}(z,w)-\varepsilon(x,y) y\cdot P_{x}(z,w)\\
    &=&F_1(x\cdot y,z,w)+\varepsilon(z,w) F_1(x\cdot y,w,z)+\varepsilon(x+y, z+w)F_2(z,w,x\cdot y)\\
    &&-x\cdot\{F_1(y,z,w)+\varepsilon(z,w)F_1(y,w,z)+\varepsilon(y,z+w)F_2(z,w,y)\}\\
    &&-\varepsilon(x,y)y\cdot\{F_1(x,z,w)+\varepsilon(z,w)F_1(x,w,z)+\varepsilon(x, z+w)F_2(z,w,x)\}\\
    &=&\big\{ F_1(x\cdot y,z,w)-x\diamond F_1(y,z,w)-\varepsilon(x,y)y\diamond F_1(x,z,w)\big\}\\
    &&+\big\{ \varepsilon(z,w)F_1(x\cdot y,w,z)-\varepsilon(z,w)x\diamond F_1(y,w,z)-\varepsilon(x,y)\varepsilon(z,w)y\diamond F_1(x,w,z)\big\}\\
    &&+\big\{\varepsilon(x+y, z+w)F_2(z,w,x\diamond y)-\varepsilon(x,y)\varepsilon(y,x+z+w)F_1(x,z,w)\diamond y\\
    &&-\varepsilon(x,y)\varepsilon(z,w)\varepsilon(y,x+w+z)F_1(x,w,z)\diamond y
      -\varepsilon(x,y)\varepsilon(x,z+w)\varepsilon(y,x+w+z)F_2(z,w,x)\diamond y\\
    &&-\varepsilon(y,z+w)x\diamond F_2(z,w,y)\big\}
      +\big\{\varepsilon(x+y, z+w)\varepsilon(x, y)F_2(z,w,y\diamond x)\\
    &&-\varepsilon(x,y+z+w)F_1(y,z,w)\diamond x-\varepsilon(z, w)\varepsilon(x, y+z+w)F_1(y,w,z)\diamond x\\
    &&-\varepsilon(y, z+w)\varepsilon(x, z+w+y)F_2(z,w,y)\diamond x-\varepsilon(x, y)\varepsilon(x, z+w)y\diamond F_2(z,w,x)\big\}\\
    &=&\varepsilon(x+y, z+w)\big\{F_2(z,w,x\diamond y)-\varepsilon(z+w, x)F_1(x,z,w)\diamond y\\
    &&-\varepsilon(z,w)\varepsilon(z+w, x)F_1(x,w,z)\diamond y-F_2(z,w,x)\diamond y-\varepsilon(z+w,x)x\diamond F_2(z,w,y)\big\}\\
    &&+\varepsilon(x,y+z+w)\big\{\varepsilon(y, z+w)F_2(z,w,y\diamond x)-F_1(y,z,w)\diamond x-\varepsilon(z, w)F_1(y,w,z)\diamond x\\
    &&-\varepsilon(y, z+w)F_2(z,w,y)\diamond x-y\diamond F_2(z,w,x)\big\}\\
    &=&\varepsilon(y, z+w)\big\{\varepsilon(x, z+w)F_2(z,w,x\diamond y)-F_1(x,z,w)\diamond y\\
    &&-\varepsilon(z,w)F_1(x,w,z)\diamond y-\varepsilon(x, z+w)F_2(z,w,x)\diamond y-x\diamond F_2(z,w,y)\big\}\\
    &=& 0.
  \end{eqnarray*}
  Hence, $(A,\cdot,[ , ],\varepsilon)$ is an $F$-manifold color algebra.

(2)\     $(A,\frkL)$ is a representation of the $\varepsilon$-commutative associative algebra $(A,\cdot, \varepsilon)$ by Lemma \ref{lem4.2}. It is known that $(A, L)$ is a representation of the  Lie color algebra $(A, [ , ], \varepsilon)$.  Note that $F_1(x,y,z)=R_{L,\frkL}(x,y)(z)$, thus the equation
$$F_1(x\cdot y,z,w)=x\diamond F_1(y,z,w)+\varepsilon(x,y)y\diamond F_1(x,z,w)$$
implies  $$R_{L,\frkL}(x\cdot y,z)=\frkL_{x}R_{L,\frkL}(y,z)+\varepsilon(x,y)\frkL_{y}R_{L,\frkL}(x,z).$$
On the other hand,
 $F_2(x,y,z)=S_{L,\frkL}(x,y)(z)$, thus combining (\ref{equ}), the equation
  \begin{eqnarray*}
    &&(F_1(x,y,z)+\varepsilon(y, z)F_1(x,z,y)+\varepsilon(x, y+z)F_2(y,z,x))\diamond w\\
    &=&\varepsilon(x, y+z)F_2(y,z,x\diamond w)-x\diamond F_2(y,z,w)
  \end{eqnarray*}
  implies $$ \frkL_{P_{x}(y,z)}=\varepsilon(x,y+z)S_{L,\frkL}(y,z)\frkL_{x}-\frkL_{x} S_{L,\frkL}(y,z).$$
 Hence,  $(A;L,\frkL)$ is a representation of $(A,\cdot,[ , ],\varepsilon)$.
\end{proof}

\section{Acknowledgements}
ZC was partially supported by National Natural Science Foundation of China (11931009) and Natural Science Foundation of Tianjin (19JCYBJC30600).
J. Li was partially supported by National natural Science Foundation of China (11771331).

 \end{document}